\pdfoutput=1
\documentclass[pdflatex,sn-mathphys-num]{sn-jnl}

\usepackage[utf8]{inputenc}
\usepackage[T1]{fontenc}
\usepackage{lmodern}
\usepackage{microtype}
\usepackage{amsmath,amssymb,amsfonts}
\usepackage{amsthm,mathtools}
\usepackage{mathrsfs}
\usepackage{enumitem}

\hypersetup{
  colorlinks = true,
  linkcolor  = blue,
  citecolor  = blue,
  urlcolor   = blue,
  bookmarksdepth = 3,
  pdftitle   = {Hypersurfaces containing involutive cones in projective symplectic spaces},
  pdfauthor  = {Gabriel A. Guedes},
  pdfsubject = {Algebraic geometry; MSC 14N10, 14N05, 14C17, 53D10},
  pdfkeywords = {involutive varieties, symplectic polarity, hypersurfaces containing cones,
    enumerative geometry, incidence varieties, Segre classes}
}

\theoremstyle{thmstyleone}
\newtheorem{theorem}{Theorem}[section]
\newtheorem{proposition}[theorem]{Proposition}
\newtheorem{lemma}[theorem]{Lemma}
\newtheorem{corollary}[theorem]{Corollary}

\theoremstyle{thmstylethree}
\newtheorem{definition}[theorem]{Definition}

\theoremstyle{thmstyletwo}
\newtheorem{remark}[theorem]{Remark}

\DeclareMathOperator{\rk}{rk}

\DeclareMathOperator{\codim}{codim}

\DeclareMathOperator{\Sym}{Sym}
\DeclareMathOperator{\Sp}{Sp}

\newcommand{\PP}{\mathbb P}
\newcommand{\CC}{\mathbb C}
\newcommand{\calF}{\mathcal F}
\newcommand{\calG}{\mathcal G}
\newcommand{\calQ}{\mathcal Q}
\newcommand{\calO}{\mathcal O}
\newcommand{\dualPP}{\check{\mathbb P}}

\begin{document}

\title[Hypersurfaces containing involutive cones]{Hypersurfaces containing involutive cones
in projective symplectic spaces}

\author*[1]{\fnm{Gabriel A.} \sur{Guedes}}\email{gabriel.guedes@ufrpe.br}

\affil*[1]{\orgdiv{Departamento de Matem\'atica},
\orgname{Universidade Federal Rural de Pernambuco},
\orgaddress{\street{Rua Dom Manoel de Medeiros, s/n, Dois Irm\~aos},
\city{Recife}, \postcode{52171-900}, \state{Pernambuco}, \country{Brazil}};
\href{https://orcid.org/0009-0008-8793-5250}{ORCID: 0009-0008-8793-5250}}

\abstract{Let $V$ be a complex symplectic vector space and let $\PP(V)$ carry the induced
contact structure and symplectic polarity. We study the loci of hypersurfaces containing
codimension-two cones supported on hyperplanes. The starting point is the following
characterization: a projective subvariety $X\subset \PP(V)$ which is a hypersurface of
degree at least two in a hyperplane $H$ is involutive if and only if it is a cone whose
vertex contains the polar point $\sigma(H)$. We give a short proof, construct the
corresponding incidence spaces, compute their dimensions and express the incidence degrees
as Segre-class integrals. In $\PP^3$ the incidence map is birational for every $d\ge m\ge2$
except $(m,d)=(2,2)$, and the same holds in $\PP^5$; the quadratic case is exceptional in
every dimension, the incidence having generic degree $2n$ in $\PP^{2n-1}$. We also give a
sufficient criterion for birationality in higher dimension, closed formulas in $\PP^3$ and
coefficient formulas in $\PP^5$. For quadric cones in $\PP^5$ the degree is the product of
shifted binomial factors and an irreducible polynomial of degree $36$, and we explain where
the shifted factors come from. All formulas and tables are checked by exact computer
algebra.}



\maketitle

\section{Introduction}

Let $(V,\omega)$ be a complex symplectic vector space of dimension $2n$.  Although the
odd-dimensional projective space $\PP(V)$ is not itself symplectic, the form $\omega$
induces on it a contact distribution and a projective polarity.  A projective subvariety is
called involutive when the homogeneous ideal of its affine cone is closed under the Poisson
bracket associated with $\omega$.  Such varieties arose as characteristic varieties of
modules over Weyl algebras; see
\cite{BernsteinLunts1988,CoutinhoLevcovitz1997,Coutinho2000,LevcovitzMcCune2002,
AlmeidaCoutinho2005,CoutinhoSaccomori2016}, and \cite{McCune2001} for the nonexistence of
involutive curves on a generic cubic surface.

For an irreducible homogeneous involutive variety contained in a hyperplane, Coutinho's
Lemma~3.1 \cite{Coutinho2000} gives, after a symplectic change of coordinates, a
decomposition $X_0\times\CC\times\{0\}$ in which the free factor is the symplectic orthogonal
of the hyperplane; projectively, the variety is a cone whose vertex contains the polar point
of the hyperplane.  The same holds for reducible varieties, since the components of an
involutive variety are involutive \cite[p.~197]{Coutinho2000}.  In $\PP^3$ this says that an
involutive plane curve is a union of lines through the polar point of its plane
\cite{LevcovitzVainsencher2011}.  The converse, in codimension two and any odd dimension,
was proved in the author's thesis \cite[Proposition~3.1.3, pp.~51--52]{Guedes2014}; see also
\cite[Proposition~1.15 and Propositions~2.19, 2.25]{Medeiros2012} for a coordinate treatment
and for the case of $\PP^3$.  Section~3 states the resulting characterization and gives a
short proof; everything else in the paper rests on it.  Note that being a cone is not enough:
the symplectic structure prescribes where the vertex has to be.

The enumerative question is then: what is the degree of the locus of hypersurfaces of
degree $d$ containing an involutive cone of degree $m$?  We answer it by building a
projective bundle over the parameter space of such cones, first with an arbitrary marked
vertex and then with the vertex fixed by the polarity, and computing a top Segre class.
Hypersurfaces containing prescribed curves, cones and other subvarieties have been counted
in the same way in
\cite{LVX2007,MeirelesXavierRojas2010,MaiaVainsencher2011,MSVX2013,CLV2014,Vainsencher2014};
we follow the conventions of \cite{Fulton1998,EisenbudHarris2016}, and refer to
\cite{AraujoVainsencher2001,EdidinGraham1998a,EdidinGraham1998b,EllingsrudStromme1996} for
the localization methods often used in such counts.  For the enumeration of contact curves
see \cite{Muratore2023Torus,Muratore2023Graph}.

The codimension-two converse, the parameter spaces and several computations in low
dimension first appeared in the author's thesis \cite{Guedes2014}.  Here we give a complete
account of that material and go further.  Section~2 recalls involutivity and the symplectic
polarity, and Section~3 proves the characterization.  Sections~4 and~5 construct the
incidence space of hypersurfaces containing a cone with a marked vertex and express its
degree as a Segre-class integral, keeping track of the generic degree of the incidence map;
Section~6 imposes the polarity.  Section~7 shows that in $\PP^3$ the incidence is birational
for all $d\ge m\ge2$ except $(m,d)=(2,2)$, and identifies the quadratic exception, in every
odd dimension, with the eigenline incidence of Hamiltonian endomorphisms, of generic degree
$2n$.  Section~8 gives a sufficient criterion for birationality in every dimension, treats
the boundary case $d=m$ directly and determines the birational range in $\PP^5$ completely.
Sections~9 and~10 give closed formulas in $\PP^3$ and coefficient formulas in $\PP^5$; for
quadric cones in $\PP^5$ the degree factors as a product of shifted binomial factors and an
irreducible polynomial of degree $36$, and we explain the shifted factors in general.
Section~11 describes the computer verification of all identities, and Section~12 collects
some remarks.

\section{Involutive varieties and symplectic polarity}

Throughout the paper we work over $\CC$ and assume $n\ge2$.  Let
$V=\CC^{2n}$, with coordinates
$$
(x_1,\ldots,x_n,y_1,\ldots,y_n),
$$
and standard symplectic form
$$
\omega=\sum_{i=1}^n dx_i\wedge dy_i.
$$
The associated Poisson bracket on $\CC[V]$ is
$$
\{f,g\}=\sum_{i=1}^n\left(
\frac{\partial f}{\partial x_i}\frac{\partial g}{\partial y_i}
-
\frac{\partial f}{\partial y_i}\frac{\partial g}{\partial x_i}
\right).
$$

\begin{definition}
Let $X\subset \PP(V)$ be a projective algebraic subvariety, possibly reducible, and let
$I(X)\subset \CC[V]$ be its saturated homogeneous ideal.  We call $X$ \emph{involutive} if
$$
\{I(X),I(X)\}\subset I(X).
$$
Equivalently, the affine cone over $X$ is coisotropic at its smooth points.
\end{definition}

\begin{definition}[Involutive linear subspaces]
\label{def:linear-involutive}
A vector subspace $U\subset V$ is \emph{involutive} (or coisotropic) if $U^\perp\subset
U$.  Its projectivization $\PP(U)$ is then an involutive projective linear subspace.
\end{definition}

\begin{remark}
Since $\dim U+\dim U^\perp=2n$, an involutive subspace satisfies $\dim U\ge n$.  Equality
holds precisely when $U=U^\perp$, that is, when $U$ is Lagrangian.
\end{remark}

A hyperplane $H=V(\ell)\subset \PP(V)$ determines a unique vector $v_\ell\in V$, up to
scalar, by
$$
\ell(u)=\omega(v_\ell,u)\qquad(u\in V).
$$
The resulting isomorphism
$$
\sigma:\dualPP(V)\longrightarrow \PP(V),\qquad H\longmapsto[v_\ell],
$$
is the \emph{symplectic polarity}; we call $\sigma(H)$ the polar point of $H$.  Since
$\omega(v_\ell,v_\ell)=0$, one has $\sigma(H)\in H$.  After a symplectic change of
coordinates, any hyperplane has the form
$$
H=V(y_n),\qquad \sigma(H)=[0:\cdots:0:1:0:\cdots:0],
$$
where the nonzero coordinate is $x_n$.

Coutinho's Lemma~3.1 gives a product decomposition for homogeneous involutive varieties in a
hyperplane \cite{Coutinho2000}.  Applied to affine cones and written in the coordinates above,
it yields the following statement.

\begin{theorem}[Coutinho's product decomposition]
\label{thm:coutinho-product}
Let $X\subset\PP(V)$ be an irreducible involutive projective variety contained in a
hyperplane $H$.  After a symplectic change of coordinates taking $H$ to $V(y_n)$, the affine
cone $\widehat X$ has the form
$$
\widehat X=X_0\times\CC\times\{0\}
\subset \CC^{2n-2}\times\CC_{x_n}\times\CC_{y_n},
$$
where $X_0\subset\CC^{2n-2}$ is irreducible, homogeneous, and involutive.  In particular,
$X$ is a cone whose vertex contains $\sigma(H)$.
\end{theorem}

Coutinho states the lemma for irreducible varieties; since the irreducible components of an
involutive variety are involutive \cite[p.~197]{Coutinho2000}, the conclusion holds
componentwise.  For plane curves this is the union-of-lines statement of
\cite{LevcovitzVainsencher2011}; see also \cite[Proposition~1.15 and Propositions~2.19,
2.25]{Medeiros2012}.  The converse, in codimension two, is
\cite[Proposition~3.1.3, pp.~51--52]{Guedes2014}.  We prove both directions together in the
next section.

\section{The polar-vertex characterization in codimension two}

\begin{definition}
Let $H\simeq\PP^r$ be a hyperplane, let $p\in H$, and let $X=V_H(F)\subset H$ be a
projective hypersurface.  We say that $X$ is a \emph{cone whose vertex contains $p$} if,
after choosing homogeneous coordinates $[z_0:\cdots:z_r]$ on $H$ with
$p=[1:0:\cdots:0]$, the equation $F$ is independent of $z_0$.  Equivalently, every line
joining $p$ to a point of $X$ is contained in $X$.
\end{definition}

\begin{theorem}[Polar-vertex characterization in codimension two]
\label{thm:cone-characterization}
Let $X\subset \PP^{2n-1}$ be a projective subvariety, possibly reducible, of codimension
two.  Assume that $X$ is contained in a hyperplane $H$ and has degree $m\ge2$ as a
hypersurface of $H$.  Then the following are equivalent:
\begin{enumerate}[label=\textup{(\roman*)}]
\item $X$ is involutive;
\item $X$ is a cone whose vertex contains the polar point $\sigma(H)$.
\end{enumerate}
\end{theorem}

\begin{proof}
The two conditions and their relative incidence are invariant under the action of $\Sp(V)$,
because
$$
g\bigl(\sigma(H)\bigr)=\sigma(gH),\qquad g\in\Sp(V).
$$
We may therefore assume that $H=V(y_n)$, so that $\sigma(H)$ is the $x_n$-coordinate
point.  Since $X$ is a hypersurface in $H$, its homogeneous ideal is
$$
I(X)=(y_n,F),
$$
where the class of $F$ modulo $y_n$ is a nonzero degree-$m$ form.  Replacing $F$ by
that representative, we assume that $F$ is independent of $y_n$.

Suppose first that $X$ is involutive.  Then
$$
\{y_n,F\}=-\frac{\partial F}{\partial x_n}\in(y_n,F).
$$
The left-hand side has degree $m-1$, whereas $F$ has degree $m$.  Hence the
degree-$(m-1)$ part of the ideal contributes only multiples of $y_n$, and
$$
\frac{\partial F}{\partial x_n}\in(y_n).
$$
Because this derivative is independent of $y_n$, it vanishes.  In characteristic zero, $F$
is therefore independent of $x_n$, so $X$ is a cone whose vertex contains $\sigma(H)$.

Conversely, if the vertex contains $\sigma(H)$, the equation $F$ is independent of
$x_n$.  Thus
$$
\{y_n,F\}=-\frac{\partial F}{\partial x_n}=0,
$$
and of course $\{y_n,y_n\}=\{F,F\}=0$.  Since the Poisson bracket is a biderivation,
$(y_n,F)$ is involutive.
\end{proof}

\begin{remark}
The implication \textup{(i)}$\Rightarrow$\textup{(ii)} is the codimension-two case of
Coutinho's product decomposition \cite[Lemma~3.1]{Coutinho2000}, and
\textup{(ii)}$\Rightarrow$\textup{(i)} is \cite[Proposition~3.1.3, pp.~51--52]{Guedes2014};
the proof above treats both at once, for a possibly reducible $X$.  The hypothesis $m\ge2$
is not used: for $m=1$ the argument shows that the linear subspace $V(y_n,F)$ is involutive
if and only if $\partial F/\partial x_n=0$, as in Definition~\ref{def:linear-involutive}.
We keep it because only cones of degree at least two occur below.
\end{remark}

\begin{remark}[Why codimension two is essential]
The statement fails in larger codimension.  In $\PP^5$, with coordinates
$(x_1,x_2,x_3,y_1,y_2,y_3)$, let
$$
H=V(y_3),\qquad \sigma(H)=[0:0:1:0:0:0],
$$
and take
$$
X=V(y_3,x_1,y_1).
$$
This linear cone has vertex containing $\sigma(H)$, but
$$
\{x_1,y_1\}=1\notin I(X),
$$
so it is not involutive.
\end{remark}

\section{The incidence of hypersurfaces containing cones}

We use the convention that $\PP(E)$ parametrizes one-dimensional subspaces of
the fibers of a vector bundle $E$. Thus, if $\pi:\PP(E)\to B$ is the projection, then
$$
\calO_{\PP(E)}(-1)\subset\pi^*E
$$
is the tautological line subbundle. All bundles below are implicitly pulled back along the
evident structure morphisms. We abbreviate $\calO_{\PP(E)}(\pm1)$ to $\calO_E(\pm1)$
whenever no confusion can arise.

Let $\calF=V^\vee\otimes \calO_{\dualPP^{2n-1}}$ be the trivial bundle of linear forms over
$\dualPP^{2n-1}$. On $\dualPP^{2n-1}$ we have the tautological exact sequence
\begin{equation}
\label{eq:dual-tautological}
0\longrightarrow \calO_{\dualPP^{2n-1}}(-1)\longrightarrow \calF\longrightarrow
\calQ\longrightarrow 0.
\end{equation}
The quotient bundle $\calQ$ has rank $2n-1$. The projective bundle $\PP(\calQ^\vee)$
parametrizes pairs $(H,p)$, where $H$ is a hyperplane and $p\in H$.

On $\PP(\calQ^\vee)$ there is an exact sequence
\begin{equation}
\label{eq:G-sequence}
0\longrightarrow \calG\longrightarrow \calQ\longrightarrow
\calO_{\calQ^\vee}(1)\longrightarrow 0.
\end{equation}
For a point $(H,p)$, the fiber $\calG_{(H,p)}$ consists of linear forms on $H$ vanishing
at $p$. Thus a point of $\PP(\Sym^m\calG)$ corresponds to a triple $(H,p,[F_m])$, where
$F_m$ is a degree $m$ form on $H$ defining a cone whose vertex contains $p$.

Set
$$
\mathbb X_m=\PP(\Sym^m\calG).
$$
Then $\mathbb X_m$ is the parameter space of hyperplane-supported cones of degree $m$,
with an arbitrary marked point contained in the vertex.

Let $d\ge m$. We use the convention $\Sym^kE=0$ for $k<0$. Over $\mathbb X_m$, the
tautological line
$$
\calO_{\Sym^m\calG}(-1)\subset \Sym^m\calG
$$
induces an inclusion
$$
\calO_{\Sym^m\calG}(-1)\otimes \Sym^{d-m}\calQ
\hookrightarrow
\Sym^d\calQ.
$$
Together with the symmetric power of \eqref{eq:dual-tautological},
$$
\calO_{\dualPP^{2n-1}}(-1)\otimes \Sym^{d-1}\calF
\hookrightarrow
\Sym^d\calF
\twoheadrightarrow
\Sym^d\calQ,
$$
we define $\mathcal F^m_d\subset \Sym^d\calF$ as the inverse image of
$$
\calO_{\Sym^m\calG}(-1)\otimes \Sym^{d-m}\calQ
\subset
\Sym^d\calQ.
$$
Equivalently, $\mathcal F^m_d$ fits into the exact sequence
\begin{equation}
\label{eq:Fd-exact}
0\longrightarrow
\calO_{\dualPP^{2n-1}}(-1)\otimes \Sym^{d-1}\calF
\longrightarrow
\mathcal F^m_d
\longrightarrow
\calO_{\Sym^m\calG}(-1)\otimes \Sym^{d-m}\calQ
\longrightarrow 0.
\end{equation}

\begin{definition}
Let $\mathbb W^m_d\subset \PP(\Sym^d V^\vee)$ be the image of the incidence projection
$$
\rho_m:\PP(\mathcal F^m_d)\longrightarrow \PP(\Sym^d V^\vee).
$$
Thus $\mathbb W^m_d$ parametrizes hypersurfaces of degree $d$ containing at least one
hyperplane-supported cone of degree $m$.
\end{definition}

\begin{proposition}
\label{prop:rank-dim-noninvolutive}
The rank of $\mathcal F^m_d$ is
$$
\rk(\mathcal F^m_d)=
\binom{d+2n-2}{2n-1}
+
\binom{d-m+2n-2}{2n-2}.
$$
Moreover,
$$
\dim \mathbb X_m=\binom{m+2n-3}{m}+4n-4.
$$
If $\rho_m$ is generically finite onto its image, then
$$
\dim \mathbb W^m_d=
\binom{d+2n-2}{2n-1}
+
\binom{d-m+2n-2}{2n-2}
+
\binom{m+2n-3}{m}+4n-5.
$$
\end{proposition}

\begin{remark}
\label{rem:corrected-ranks-thesis}
The rank formulas use $\rk\calF=2n$ and $\rk\calQ=2n-1$. They are the corrected versions
of the corresponding binomial factors in \cite[Thms. 3.1.4 and 3.1.6]{Guedes2014}, where the
printed indices are shifted by one.
\end{remark}

\begin{proof}
The rank formula follows from \eqref{eq:Fd-exact}. Since $\calF$ has rank $2n$,
$$
\rk\Sym^{d-1}\calF=\binom{d+2n-2}{2n-1}.
$$
Since $\calQ$ has rank $2n-1$,
$$
\rk\Sym^{d-m}\calQ=\binom{d-m+2n-2}{2n-2}.
$$
The base $\PP(\calQ^\vee)$ has dimension
$$
(2n-1)+(2n-2)=4n-3.
$$
The fiber of $\PP(\Sym^m\calG)\to \PP(\calQ^\vee)$ has dimension
$$
\binom{m+2n-3}{m}-1,
$$
because $\calG$ has rank $2n-2$. This gives the formula for $\dim \mathbb X_m$. Finally,
$$
\dim \PP(\mathcal F^m_d)=\dim\mathbb X_m+\rk(\mathcal F^m_d)-1,
$$
and generic finiteness gives the dimension of the image.
\end{proof}

\section{Segre-class degree formula}

Let $\delta_m(d)$ denote the generic degree of $\rho_m:\PP(\mathcal F^m_d)\to \mathbb
W^m_d$, when this map is generically finite.

\begin{theorem}
\label{thm:degree-noninvolutive}
Assume that $\rho_m$ is generically finite onto $\mathbb W^m_d$, with generic degree
$\delta_m(d)$. Then
$$
\delta_m(d)\deg(\mathbb W^m_d)
=
\int_{\mathbb X_m}s_{\dim\mathbb X_m}(\mathcal F^m_d)\cap[\mathbb X_m].
$$
In particular, if $\rho_m$ is birational, the integral is $\deg(\mathbb W^m_d)$.
\end{theorem}

\begin{proof}
Let $H=c_1(\calO_{\PP(\Sym^d V^\vee)}(1))$. By definition of degree and the projection
formula,
$$
\delta_m(d)\deg(\mathbb W^m_d)=
\int_{\PP(\mathcal F^m_d)}\rho_m^*H^{\dim \mathbb W^m_d}.
$$
The last integral is, by the definition of Segre classes of a projective bundle, equal to
$$
\int_{\mathbb X_m}s_{\dim\mathbb X_m}(\mathcal F^m_d)\cap[\mathbb X_m].
$$
\end{proof}

For computations, let
$$
h=c_1(\calO_{\dualPP^{2n-1}}(1)),\qquad
\ell=c_1(\calO_{\calQ^\vee}(1)),\qquad
x=c_1(\calO_{\Sym^m\calG}(1)).
$$
From \eqref{eq:Fd-exact}, using $\calF$ trivial, one obtains
\begin{equation}
\label{eq:segre-class}
s(\mathcal F^m_d)=
\frac{1}{(1-h)^{e_1}}
\frac{(1-h-x)^{e_{m+1}}}{(1-x)^{e_m}},
\end{equation}
where
$$
e_i=\rk(\Sym^{d-i}\calF)=\binom{d-i+2n-1}{2n-1}.
$$
In particular, the convention above gives $e_{m+1}=0$ when $d=m$.
The integral is computed modulo the standard projective-bundle relations. In particular,
$$
h^{2n}=0,
$$
and, on $\PP(\calQ^\vee)$,
\begin{equation}
\label{eq:second-relation}
\sum_{i=0}^{2n-1}(-1)^i h^{2n-1-i}\ell^i=0.
\end{equation}
The third relation comes from
$$
\calO\hookrightarrow \Sym^m\calG\otimes \calO_{\Sym^m\calG}(1),
$$
namely
\begin{equation}
\label{eq:third-relation}
c_{\rk(\Sym^m\calG)}(\Sym^m\calG\otimes \calO_{\Sym^m\calG}(1))=0.
\end{equation}

\section{The involutive cone incidence}

We now impose involutivity. By Theorem \ref{thm:cone-characterization}, it is equivalent to
require that the vertex of the cone contain the symplectic polar point of the supporting
hyperplane.

Define
$$
s_\sigma:\dualPP^{2n-1}\longrightarrow \PP(\calQ^\vee),\qquad
H\longmapsto (H,\sigma(H)).
$$
Let $\pi_m:\mathbb X_m\to\PP(\calQ^\vee)$ be the structure morphism. The compactified
parameter space for involutive cones is the fiber product
$$
\mathbb Y_m
=\dualPP^{2n-1}\times_{\PP(\calQ^\vee)}\mathbb X_m
=s_\sigma^*\mathbb X_m.
$$
Denote by
$$
\bar s_\sigma:\mathbb Y_m\longrightarrow\mathbb X_m
$$
the second projection. Equivalently,
$$
\mathbb Y_m=\PP(\Sym^m\calG_\sigma),
$$
where $\calG_\sigma=s_\sigma^*\calG$. We have
$$
\dim \mathbb Y_m=\binom{m+2n-3}{m}+2n-2.
$$
\begin{remark}
Let $\mathbb Y_m^{\mathrm{sf}}\subset\mathbb Y_m$ be the dense open subset parametrizing
square-free degree-$m$ forms. By Theorem~\ref{thm:cone-characterization}, its points
correspond to reduced involutive codimension-two subvarieties. The projective bundle
$\mathbb Y_m$ also contains nonreduced cones; we need the whole bundle in order to integrate.
\end{remark}
Let
$$
\widetilde{\mathcal F}^m_d=\bar s_\sigma^*\mathcal F^m_d
$$
be the pulled-back incidence bundle over $\mathbb Y_m$. Define $\mathbb B^m_d\subset
\PP(\Sym^dV^\vee)$ to be the image of
$$
\widetilde\rho_m:\PP(\widetilde{\mathcal F}^m_d)\longrightarrow \PP(\Sym^dV^\vee).
$$
Since
$$
\PP(\widetilde{\mathcal F}^m_d)|_{\mathbb Y_m^{\mathrm{sf}}}
\subset\PP(\widetilde{\mathcal F}^m_d)
$$
is dense and $\widetilde\rho_m$ is proper, $\mathbb B^m_d$ is the closure of the locus of
hypersurfaces containing a reduced involutive cone of degree $m$, in the sense of
Theorem~\ref{thm:cone-characterization}; points of the boundary may come from nonreduced
cones.

\begin{proposition}
\label{prop:dimension-involutive}
If $\widetilde\rho_m$ is generically finite onto its image, then
$$
\dim \mathbb B^m_d=
\binom{d+2n-2}{2n-1}
+
\binom{d-m+2n-2}{2n-2}
+
\binom{m+2n-3}{m}
+2n-3.
$$
\end{proposition}

\begin{proof}
The rank of $\widetilde{\mathcal F}^m_d$ is the same as the rank of $\mathcal F^m_d$.
Since
$$
\dim \mathbb Y_m=\binom{m+2n-3}{m}+2n-2,
$$
we have
$$
\dim\PP(\widetilde{\mathcal F}^m_d)
=
\dim\mathbb Y_m+\rk(\widetilde{\mathcal F}^m_d)-1.
$$
Generic finiteness gives the stated dimension.
\end{proof}

\begin{theorem}
\label{thm:degree-involutive}
Assume that $\widetilde\rho_m$ is generically finite onto its image, with generic degree
$\widetilde\delta_m(d)$. Then
$$
\widetilde\delta_m(d)\deg(\mathbb B^m_d)
=
\int_{\mathbb Y_m}s_{\dim\mathbb Y_m}(\widetilde{\mathcal F}^m_d)\cap[\mathbb Y_m].
$$
In particular, if $\widetilde\rho_m$ is birational, this integral is the degree of $\mathbb
B^m_d$.
\end{theorem}

\begin{proof}
The proof is identical to that of Theorem \ref{thm:degree-noninvolutive}, replacing $\mathbb
X_m$ and $\mathcal F^m_d$ by their pullbacks $\mathbb Y_m$ and $\widetilde{\mathcal
F}^m_d$.
\end{proof}

Whether $\widetilde\rho_m$ is generically finite can be read off from the integral itself.

\begin{corollary}
\label{cor:nonvanishing-finite}
Let $N=\dim\PP(\widetilde{\mathcal F}^m_d)$ and $H=c_1(\calO_{\PP(\Sym^dV^\vee)}(1))$.  If
$$
\int_{\mathbb Y_m}s_{\dim\mathbb Y_m}(\widetilde{\mathcal F}^m_d)\cap[\mathbb Y_m]
=\int_{\PP(\widetilde{\mathcal F}^m_d)}\widetilde\rho_m^{\,*}H^{N}\neq0,
$$
then $\widetilde\rho_m$ is generically finite onto $\mathbb B^m_d$, and the integral equals
$\widetilde\delta_m(d)\deg(\mathbb B^m_d)$.
\end{corollary}

\begin{proof}
If $\dim\mathbb B^m_d<N$, then $\widetilde\rho_{m*}[\PP(\widetilde{\mathcal F}^m_d)]=0$ in
$A_N(\PP(\Sym^dV^\vee))$, and the integral vanishes by the projection formula.  The last
assertion is Theorem~\ref{thm:degree-involutive}.
\end{proof}

\begin{remark}[The incidence degree]
The incidence projection need not be birational, which is why $\widetilde\delta_m(d)$
appears in the formula. In $\PP^3$ it is birational for all $d\ge m$ except $(m,d)=(2,2)$,
where the image is the whole space of quadrics and the map has degree four
(Section~\ref{sec:P3-birationality}); the same holds in $\PP^5$
(Section~\ref{sec:higher-birationality}).
\end{remark}

\section{Generic degree in \texorpdfstring{$\PP^3$}{P3}}
\label{sec:P3-birationality}

Let $n=2$, so that $\mathbb Y_m$ is the parameter space of involutive plane cones of degree
$m$.  A point is a pair $(H,[q])$, where $H\subset\PP^3$ is a plane and $q$ is a binary
form of degree $m$ in the two linear coordinates of $H$ vanishing at $\sigma(H)$, that is,
the equation of a cone in $H$ whose vertex contains $\sigma(H)$.  Thus
$\dim\mathbb Y_m=m+3$.  For $C\in\mathbb Y_m$, put
$$
L_C=H^0(\PP^3,I_C(d))
$$
and
\begin{equation}
\label{eq:a-dm}
a(d,m)=\codim L_C
=\binom{d+2}{2}-\binom{d-m+2}{2}
=\frac{m(2d-m+3)}2.
\end{equation}

Let $S=\CC[x_0,x_1,x_2,x_3]$.  For two members $C,C'\in\mathbb Y_m$, write
$$
H_{C\cap C'}(d)=\dim_\CC\bigl(S/(I_C+I_{C'})\bigr)_d.
$$
This is the Hilbert function of the ideal $I_C+I_{C'}$ itself, not of its saturation.

\begin{lemma}
\label{lem:linear-codim-estimate}
For $C,C'\in\mathbb Y_m$, one has
$$
\codim_{L_C}(L_C\cap L_{C'})=a(d,m)-H_{C\cap C'}(d),
$$
where the intersection is scheme-theoretic.
\end{lemma}

\begin{proof}
Put $V_d=S_d$.  Since
$$
L_C+L_{C'}=(I_C+I_{C'})_d,
$$
we have
$$
\codim_{V_d}(L_C+L_{C'})=\dim(S/(I_C+I_{C'}))_d=H_{C\cap C'}(d).
$$
The assertion follows from $\codim_{V_d}L_{C'}=a(d,m)$.
\end{proof}

We shall use the elementary function
$$
B(t)=\max\{t+1,0\}.
$$
Put $R=\CC[u,v]$. If $f,g\in R$ are nonzero binary forms of degree $m$ and their greatest
common divisor $D$ has degree $r$, write $f=Da$ and $g=Db$. Since $a$ and $b$ are
coprime, the elementary graded resolution
$$
0\longrightarrow R(-(2m-r))
\xrightarrow{\,(-b,a)\,}
R(-m)^2
\xrightarrow{\,(f,g)\,}
R
\longrightarrow R/(f,g)\longrightarrow0
$$
gives
\begin{equation}
\label{eq:binary-Hilbert}
H_{\CC[u,v]/(f,g)}(d)
=B(d)-2B(d-m)+B(d-2m+r).
\end{equation}
For $d\ge m$, this is
\begin{equation}
\label{eq:binary-Hilbert-max}
\max\{r,2m-d-1\}.
\end{equation}

\begin{theorem}
\label{thm:P3-birationality}
Let $m\ge2$ and $d\ge m$.  If $(m,d)\ne(2,2)$, the incidence projection
$$
\widetilde\rho_m:\PP(\widetilde{\mathcal F}^m_d)\longrightarrow\mathbb B^m_d
$$
is birational.
\end{theorem}

\begin{proof}
Fix a general $C=(H,[q])\in\mathbb Y_m$, with $q$ square-free, and consider a general form $F\in
L_C$.  Write
$$
F=hA+qB,
$$
where $h$ defines $H$ and $B$ is a degree-$(d-m)$ form on $H$.  For general $B$,
no nonconstant binary factor through $\sigma(H)$ divides $B$.  Hence a second degree-$m$
cone supported on the same plane would have equation $q'\mid qB$, forcing $q'\mid q$ and
therefore $[q']=[q]$.

Let now $C'=(H',[q'])$ with $H'\ne H$, and set $L=H\cap H'$. We show, stratum by
stratum, that the locus of forms in $L_C$ containing such a $C'$ has proper closure in
$L_C$.

\smallskip
\noindent\emph{(a) Neither cone contains $L$, and $\sigma(H')\notin L$.}
The restriction map from the binary forms defining cones in $H'$ to $H^0(L,\calO_L(m))$ is
an isomorphism.  Let $r$ be the degree of the gcd of $q|_L$ and $q'|_L$.  For fixed
$H'$, the forms with this overlap have projective dimension at most $m-r$; hence the
stratum has dimension at most $m+3-r$.  By \eqref{eq:binary-Hilbert-max} and
Lemma~\ref{lem:linear-codim-estimate}, its additional codimension inside $L_C$ is at least
$$
a(d,m)-\max\{r,2m-d-1\}.
$$
If the maximum is $r$, properness follows from $a(d,m)>m+3$.  If it is $2m-d-1$, the
required difference is
$$
a(d,m)+d-3m-2+r.
$$
For $m\ge3$ this is already positive at $d=m$ and $r=0$, and it increases with $d$; for
$m=2$ it equals $3d-7+r$, which is positive for every $d\ge3$.  The only failure is
$(m,d)=(2,2)$.

\smallskip
\noindent\emph{(b) Neither cone contains $L$, and $\sigma(H')\in L$.}
The planes $H'$ satisfying this extra incidence form a family of dimension at most two,
while $[q']$ varies in $\PP^m$.  Thus the stratum has dimension at most $m+2$.  On
$L$, the restriction of $q'$ is a pure $m$-th power.
Formula~\eqref{eq:binary-Hilbert-max} therefore gives $H_{C\cap C'}(d)\le m$, so the
additional codimension is at least $a(d,m)-m$.  The inequality
$$
a(d,m)-m>m+2
$$
holds for every $d\ge m$, except $(m,d)=(2,2)$.

\smallskip
\noindent\emph{(c) Exactly one of the cones contains $L$.}
This stratum has dimension at most $m+1$: if $C$ contains $L$, the line is one of the
finitely many components of the general cone and $H'$ varies in a pencil; if $C'$ contains
$L$, the condition $\sigma(H')\in L$ and divisibility by the line give the same bound.
The intersection on $L$ is a divisor of degree $m$, so the additional codimension is at
least $a(d,m)-m$.  Since
$$
a(d,m)-m>m+1
$$
away from $(2,2)$, this stratum does not dominate.

\smallskip
\noindent\emph{(d) Both cones contain $L$.}
The line $L$ is then one of the finitely many components of $C$.  The plane $H'$ varies
in a pencil and the residual binary form has projective dimension $m-1$, so the stratum has
dimension at most $m$.  Moreover
$$
I_C+I_{C'}=(h,h'),
$$
hence $C\cap C'=L$ scheme-theoretically and $H_{C\cap C'}(d)=d+1$.  The inequality
$$
a(d,m)-(d+1)>m
$$
holds for $d\ge m$, except when $(m,d)=(2,2)$.

Thus a general $F\in L_C$ contains no second cone in the family.  The incidence space is
irreducible, being a projective bundle over the irreducible variety $\mathbb Y_m$.  The generic
fiber over $\mathbb B^m_d$ therefore consists of one point; in characteristic zero the
induced extension of function fields has degree one.  Hence $\widetilde\rho_m$ is
birational.
\end{proof}

For $m=d=2$ the incidence can be read in terms of Lie algebras: under the standard
correspondence $\Sym^2V^\vee\simeq\mathfrak{sp}(V)$ \cite{Buczynski2006} it becomes the
eigenline incidence of Hamiltonian endomorphisms, a special case of the eigenvector varieties
of \cite{DiRoccoSturmfelsSverrisdottir2026}.

\begin{proposition}[Quadrics in arbitrary odd projective dimension]
\label{prop:quadratic-general}
Let $n\ge2$.  For $m=d=2$, the image of the involutive cone incidence is the full
projective space of quadrics in $\PP^{2n-1}$, and the generic degree of the incidence map is
$2n$.
\end{proposition}

\begin{proof}
The symplectic form induces a linear isomorphism
$$
\Sym^2V^\vee\longrightarrow\mathfrak{sp}(V),\qquad
Q\longmapsto A_Q,\qquad B_Q(u,v)=\omega(A_Qu,v),
$$
where $B_Q$ is the symmetric bilinear form polarizing $Q$.
Let $U_{\mathrm{rs}}\subset\PP(\Sym^2V^\vee)$ be the dense open subset corresponding to
regular semisimple elements of $\mathfrak{sp}(V)$.  If $Q\in U_{\mathrm{rs}}$, then the
associated Hamiltonian endomorphism $A_Q$ has $2n$ distinct nonzero eigenvalues and hence
$2n$ eigenlines.  If $[p]$ is an eigenline and $H=p^\perp$, then
$$
B_Q(p,v)=\omega(A_Qp,v)=0\qquad(v\in H).
$$
Since $p\in p^\perp=H$, the same vanishing gives $B_Q(p,p)=0$. Thus $p\in Q\cap H$ and the
differential of $Q|_H$ vanishes at $p=\sigma(H)$. Moreover, the radical of
$B_Q|_{p^\perp}$ is
$$
\{v\in p^\perp : A_Qv\in(p^\perp)^\perp=\CC p\}=p^\perp\cap A_Q^{-1}(\CC p)=\CC p,
$$
because $A_Q$ is invertible and $A_Qp\in\CC p$. Hence $Q|_H$ has rank $2n-2\ge2$: the
hyperplane section is a reduced quadric cone whose vertex is exactly $p$, and it is
involutive by Theorem~\ref{thm:cone-characterization}.

Conversely, if $Q\cap H$ is such a cone with vertex $p=\sigma(H)$, then $B_Q(p,v)=0$ for
every $v\in p^\perp$. Consequently
$$
A_Qp\in(p^\perp)^\perp=\CC p,
$$
so $[p]$ is an eigenline of $A_Q$. More precisely, the linear map
$$
V/\CC p\longrightarrow(p^\perp)^\vee,\qquad
\overline w\longmapsto\omega(w,-)|_{p^\perp},
$$
is an isomorphism. Hence the scheme-theoretic singularity condition
is exactly
$$
B_Q(p,-)|_{p^\perp}=0
\quad\Longleftrightarrow\quad
A_Qp\bmod\CC p=0
\quad\Longleftrightarrow\quad
A_Qp\wedge p=0.
$$

The eigenline scheme of a regular semisimple $A_Q$ is reduced. Indeed, in an eigenbasis with
distinct eigenvalues
$\lambda_1,\ldots,\lambda_{2n}$, on the affine chart $z_i=1$ the equations
$A_Qz\wedge z=0$ contain
$$
(\lambda_j-\lambda_i)z_j=0\qquad(j\ne i),
$$
and hence cut out the $i$-th eigenline as a reduced point. Once $[p]$ is fixed, both
$H=p^\perp$ and $[Q|_H]$ are determined; moreover, $Q|_H\ne0$, since otherwise the
nondegenerate quadric $Q$ would contain $H$. Thus the fiber of the incidence map over every
$Q\in U_{\mathrm{rs}}$ is the eigenline scheme and consists of these $2n$ reduced points.

Finally, the incidence source is projective, so its image under the incidence projection is
closed. Since that image contains the dense open subset $U_{\mathrm{rs}}$, it is the full
projective space of quadrics. The generic degree of the incidence map is therefore $2n$.
\end{proof}

\begin{corollary}
In $\PP^3$, $\deg\mathbb B^2_2=1$ and the generic incidence degree is $4$.  In
$\PP^5$, the corresponding generic incidence degree is $6$.
\end{corollary}

\section{Birationality in higher dimension}
\label{sec:higher-birationality}

In $\PP^{2n-1}$ with $n\ge3$ two distinct hyperplanes meet along a $\PP^{2n-3}$, so two
cones with different supports no longer meet in finitely many points, and the strata of the
proof of Theorem~\ref{thm:P3-birationality} have to be replaced by an estimate of the Hilbert
function of the intersection. This gives a sufficient criterion
(Theorem~\ref{thm:higher-birationality-criterion}), which misses a few cases of low degree.
For $d=m$, and for $(m,d)=(2,3)$ in $\PP^5$, we then argue directly.

For $n\ge2$, put
$$
A_n(d,m)=
\binom{d+2n-2}{2n-2}
-
\binom{d-m+2n-2}{2n-2},
$$
so that $A_n(d,m)$ is the codimension, inside $H^0(\PP^{2n-1},\calO(d))$, of the linear
space of forms containing a fixed involutive cone of degree $m$. Also set
$$
D_n(m)=\dim \mathbb Y_m=\binom{m+2n-3}{m}+2n-2.
$$
We use the convention that $\binom{a}{b}=0$ when $a<b$. Define
\begin{equation}
\label{eq:Hn-definition}
H_n(d,m)=\binom{d+2n-3}{2n-3}-\binom{d-m+2n-3}{2n-3},
\end{equation}
the degree-$d$ value of the Hilbert function of a hypersurface of degree $m$ in
$\PP^{2n-3}$. By Pascal's identity,
\begin{equation}
\label{eq:An-minus-Hn}
A_n(d,m)-H_n(d,m)=\binom{d+2n-3}{2n-2}-\binom{d-m+2n-3}{2n-2}.
\end{equation}

\begin{theorem}[A higher-dimensional birationality criterion]
\label{thm:higher-birationality-criterion}
Let $n\ge2$, $m\ge2$, and $d\ge m$. If
\begin{equation}
\label{eq:higher-birationality-criterion}
A_n(d,m)-H_n(d,m)>D_n(m),
\end{equation}
then the involutive cone incidence
$$
\widetilde\rho_m:\PP(\widetilde{\mathcal F}^m_d)\longrightarrow \mathbb B^m_d
$$
in $\PP^{2n-1}$ is birational.
\end{theorem}

\begin{proof}
If $n=2$, then
$$
A_2(2,2)=5,\qquad H_2(2,2)=2,\qquad D_2(2)=5,
$$
so condition~\eqref{eq:higher-birationality-criterion} fails for $(m,d)=(2,2)$. Hence,
under the hypothesis of the theorem, the conclusion for $n=2$ follows from
Theorem~\ref{thm:P3-birationality}. We may therefore assume that $n\ge3$.

Fix a general involutive cone $C=(H,[q])\in \mathbb Y_m$. Write
$$
L_C=H^0(\PP^{2n-1},I_C(d)).
$$
As in the proof of Theorem~\ref{thm:P3-birationality}, a general element of $L_C$ can be
written as
$$
F=hA+qB,
$$
where $h$ is an equation of $H$ and $B$ is a general form of degree $d-m$ on $H$. If
a second involutive cone has the same supporting hyperplane $H$, then its equation is
another degree $m$ form in the subring generated by the linear forms vanishing at the vertex
$\sigma(H)$. Every factor of such a form belongs to the same subring, and for a general $B$
no nonconstant form of the subring divides $B$. Hence any such second equation must
divide $q$, and therefore agrees with $q$ up to scalar. Thus a general $F\in L_C$
contains no second cone with support $H$.

It remains to exclude cones $C'=(H',[q'])$ with $H'\ne H$. Put
$$
K=H\cap H'\simeq \PP^{2n-3}.
$$
Since $n\ge3$, the defining form $q$ of a general $C$ is a general degree-$m$ form in
$2n-2\ge4$ variables and hence has no linear factor.
Consequently $q|_K\ne0$ for every $H'\ne H$. The scheme $Z=C\cap C'$ is cut out on
$K$ by the two degree $m$ forms obtained by restricting $q$ and $q'$ to $K$. Let $s$ be
the degree of their greatest common divisor, using the convention $\gcd(f,0)=f$, so that
$0\le s\le m$. If $s<m$, then after removing this common divisor the two residual forms
have no common factor, hence form a regular sequence in the homogeneous coordinate ring of
$K$. If $s=m$, the ideal generated by the two restrictions is principal. In both cases, the
degree $d$ Hilbert function of $Z$ equals
\begin{equation}
\label{eq:Hilbert-Z}
\binom{d+2n-3}{2n-3}
-2\binom{d-m+2n-3}{2n-3}
+\binom{d-2m+s+2n-3}{2n-3},
\end{equation}
which is increasing in $s$ and is therefore bounded above by its value at $s=m$, namely
$H_n(d,m)$.

For a fixed $C'$, the same linear-algebra computation as in
Lemma~\ref{lem:linear-codim-estimate} gives
$$
\codim_{L_C}(L_C\cap L_{C'})
\ge A_n(d,m)-H_n(d,m).
$$
The family of all possible $C'$'s has dimension $D_n(m)$. Thus, if
\eqref{eq:higher-birationality-criterion} holds, the union of the linear spaces $L_C\cap
L_{C'}$, with $C'\ne C$ and $H'\ne H$, is a proper subset of $L_C$. Consequently a
general form in $L_C$ contains no involutive cone of degree $m$ other than $C$. Since
$\mathbb Y_m$ is irreducible and the incidence space is a projective bundle over it, the
generic fiber of $\widetilde\rho_m$ consists of one point. Hence $\widetilde\rho_m$ is
birational.
\end{proof}

\begin{corollary}[The case \texorpdfstring{$\PP^5$}{P5}]
\label{cor:P5-birationality}
In $\PP^5$, the involutive cone incidence is birational whenever
\begin{equation}
\label{eq:P5-birationality-inequality}
\binom{d+3}{4}-\binom{d-m+3}{4}
>
\binom{m+3}{3}+4.
\end{equation}
In particular, the criterion proves birationality for $m=2$ and $m=3$ whenever $d\ge4$,
for $m=4$ whenever $d\ge5$, and for every $m\ge5$ throughout the natural range $d\ge
m$.
\end{corollary}

\begin{proof}
For $n=3$, identity~\eqref{eq:An-minus-Hn} reads
$A_3(d,m)-H_3(d,m)=\binom{d+3}{4}-\binom{d-m+3}{4}$, and $D_3(m)=\binom{m+3}{3}+4$, so
the criterion of Theorem~\ref{thm:higher-birationality-criterion} is exactly
\eqref{eq:P5-birationality-inequality}.
The displayed ranges follow by direct evaluation at the first listed value of $d$, together
with the fact that the left-hand side of \eqref{eq:P5-birationality-inequality} is increasing
in $d$.
\end{proof}

In $\PP^5$ the criterion misses $(m,d)=(2,2)$, $(2,3)$, $(3,3)$ and $(4,4)$. The first is
a true exception, by Proposition~\ref{prop:quadratic-general}. For the other three the
estimate in the proof of Theorem~\ref{thm:higher-birationality-criterion} is too crude:
$L_C\cap L_{C'}$ always contains $hh'S_{d-2}$, the space of hypersurfaces containing
$H\cup H'$, which does not depend on $q'$, so the union of the $L_C\cap L_{C'}$ is much
smaller than $\dim\mathbb Y_m+\dim(L_C\cap L_{C'})$. The next two propositions take this into
account.

\begin{proposition}[The boundary case \texorpdfstring{$d=m$}{d=m}]
\label{prop:boundary-birational}
Let $n\ge3$ and $m\ge3$.  Then the involutive cone incidence
$\widetilde\rho_m:\PP(\widetilde{\mathcal F}^m_m)\to\mathbb B^m_m$ in $\PP^{2n-1}$ is
birational.
\end{proposition}

\begin{proof}
Fix a general $C=(H,[q])\in\mathbb Y_m$; since $n\ge3$, the form $q$ is a general form of
degree $m$ in $2n-2\ge4$ variables and is irreducible.  Here
$L_C=(h,q)_m=h\,S_{m-1}\oplus\CC q$, where $S_{m-1}$ denotes the forms of degree $m-1$ on
$\PP^{2n-1}$.  Let $F=hA+cq\in L_C$ be general and suppose that
$C'=(H',[q'])\ne C$ satisfies $F\in L_{C'}$.

If $H'=H$, then $F|_H=cq$ and $q'\mid cq$; for $c\neq0$ irreducibility gives $[q']=[q]$,
and $c=0$ confines $F$ to the proper subspace $hS_{m-1}$.

Assume $H'\ne H$ and $F\supset H'$, say $F=h'G$.  Then $F|_H=(h'|_H)(G|_H)=cq$, and since
the irreducible form $q$ is not divisible by the linear form $h'|_H$, we get $c=0$ and
$G\in(h)$.  Thus $F\in\{hh'G'\}$, a locus of dimension at most
$(2n-1)+\binom{m+2n-3}{2n-1}$, which is smaller than $\dim L_C=\binom{m+2n-2}{2n-1}+1$
because $\binom{m+2n-3}{2n-2}>2n-2$.

Assume finally $H'\ne H$ and $F\not\supset H'$.  Then $F|_{H'}\ne0$ and $[q']=[F|_{H'}]$,
so the condition on $F$ is
$$
F|_{H'}\in\Sym^m\calG_{\sigma,H'} .
$$
For fixed $H'$ this defines the kernel $\Lambda_{H'}\subset L_C$ of the linear map
$$
\varphi_{H'}:L_C\longrightarrow \Sym^m\calQ_{H'}/\Sym^m\calG_{\sigma,H'},
\qquad
R:=\binom{m+2n-3}{2n-2}=\rk\bigl(\Sym^m\calQ/\Sym^m\calG\bigr).
$$
Put $\bar h=h|_{H'}\neq0$.  In coordinates on $H'$ with $\sigma(H')=[1:0:\cdots:0]$, the
elements of $\Sym^m\calG_{\sigma,H'}$ are the forms of $z_0$-degree zero, and the
$z_0$-degree of a product is the sum of the $z_0$-degrees of the factors.

\emph{Case $\sigma(H')\notin H$.}  Then $\bar h$ has $z_0$-degree one, so no nonzero cone
form is divisible by $\bar h$, i.e.\ $\bar h\,\Sym^{m-1}\calQ_{H'}\cap\Sym^m\calG_{\sigma,H'}=0$.
Hence $\varphi_{H'}(hS_{m-1})$ has dimension $\binom{m-1+2n-2}{2n-2}=R$, $\varphi_{H'}$ is
surjective and $\codim_{L_C}\Lambda_{H'}=R$.  These $H'$ form an open subset of dimension
$2n-1$ of $\dualPP^{2n-1}$, so the union of the $\Lambda_{H'}$ has dimension at most
$\dim L_C-R+2n-1<\dim L_C$, because $R\ge\binom{2n}{2n-2}=n(2n-1)>2n-1$ for $m\ge3$.

\emph{Case $\sigma(H')\in H$, $H'\ne H$.}  Since $\sigma(H')\in H$ if and only if
$\omega(v_\ell,v_{\ell'})=0$, these $H'$ form a hyperplane of $\dualPP^{2n-1}$, of dimension
$2n-2$.  Now $\bar h\in\calG_{\sigma,H'}$, and
$\bar h\,\Sym^{m-1}\calQ_{H'}\cap\Sym^m\calG_{\sigma,H'}=\bar h\,\Sym^{m-1}\calG_{\sigma,H'}$,
so that
$$
\dim\varphi_{H'}(hS_{m-1})=\binom{m+2n-3}{2n-2}-\binom{m+2n-4}{2n-3}=\binom{m+2n-4}{2n-2}.
$$
The union of the corresponding $\Lambda_{H'}$ has dimension at most
$\dim L_C-\binom{m+2n-4}{2n-2}+2n-2<\dim L_C$, because
$\binom{m+2n-4}{2n-2}\ge\binom{2n-1}{2n-2}=2n-1$ for $m\ge3$.

Hence a general $F\in L_C$ contains no involutive cone of degree $m$ other than $C$.  Since
$\PP(\widetilde{\mathcal F}^m_m)$ is irreducible, the generic fiber of $\widetilde\rho_m$ is
a single point, and in characteristic zero $\widetilde\rho_m$ is birational.
\end{proof}

\begin{remark}
For $m=2$ one has $R=2n-1$ and the first case of the proof fails, as it must by
Proposition~\ref{prop:quadratic-general}.  For $n=2$ the same argument, with $q$ square-free,
gives the case $d=m$ of Theorem~\ref{thm:P3-birationality}.
\end{remark}

\begin{proposition}[Cubics containing an involutive quadric cone in
\texorpdfstring{$\PP^5$}{P5}]
\label{prop:P5-cubic-quadric}
Let $n=3$ and $(m,d)=(2,3)$.  Then
$\widetilde\rho_2:\PP(\widetilde{\mathcal F}^2_3)\to\mathbb B^2_3$ is birational.
\end{proposition}

\begin{proof}
Fix a general $C=(H,[q])\in\mathbb Y_2$.  In coordinates $z_0,\dots,z_4$ on $H$ with
$\sigma(H)=[1:0:0:0:0]$, the form $q=q(z_1,\dots,z_4)$ is a quadratic form of rank four; in
particular $q$ is irreducible, and $C$ is the cone with vertex $\sigma(H)$ over a smooth
quadric surface $\Sigma\subset\PP^3$.  Here $A_3(3,2)=30$ and $\dim L_C=26$, and for every
$C'=(H',[q'])\in\mathbb Y_2$ the computation of Lemma~\ref{lem:linear-codim-estimate}, whose
proof does not depend on $n$, gives
$$
\dim(L_C\cap L_{C'})=H_{C\cap C'}(3)-4 .
$$

Let $F=hA+qB\in L_C$ be general, so that $B|_H$ is a general linear form on $H$.  A cone
form $q'\in\Sym^2\calG_{\sigma,H}$ with $q'\mid F|_H=q\,B|_H$ either is irreducible, and then
$q'\mid q$ and $[q']=[q]$, or has a linear factor dividing $B|_H$, which is impossible since
$B|_H$ does not vanish at $\sigma(H)$.  Hence, away from the proper subspace $hS_2$, no
second cone is supported on $H$.

Now let $H'\ne H$, put $K=H\cap H'\simeq\PP^3$, and let $\lambda\in\calQ_H$ be an
equation of $K$ in $H$.  Since $q$ is irreducible, $q|_K\ne0$.  With $s$ the degree of the
greatest common divisor of $q|_K$ and $q'|_K$ (convention $\gcd(f,0)=f$), formula
\eqref{eq:Hilbert-Z} gives $H_{C\cap C'}(3)=12+\binom{s+2}{3}$, that is, $12$, $13$, $16$
for $s=0,1,2$, so that
$$
\dim(L_C\cap L_{C'})=8,\ 9,\ 12\qquad\text{for } s=0,1,2 .
$$
As in the proof of Theorem~\ref{thm:P3-birationality} we stratify the $C'$ with $H'\ne H$;
for each stratum we need its dimension plus $\dim(L_C\cap L_{C'})$ to be smaller than $26$.
We use three facts.
\begin{enumerate}[label=\textup{(\Alph*)}, leftmargin=2.2em]
\item $\sigma(H')\in K$ if and only if $\sigma(H)\in K$ if and only if
  $\omega(v_\ell,v_{\ell'})=0$; the hyperplanes $H'$ with this property form a hyperplane
  $\Pi_H\subset\dualPP^5$.
\item Let $\rho:\Sym^2\calG_{\sigma,H'}\to\Sym^2(\CC[K]_1)$ be the restriction to $K$.  If
  $\sigma(H')\notin K$, the projection from $\sigma(H')$ identifies $K$ with
  $\PP(\widehat{H'}/\CC\sigma(H'))$ and $\rho$ is an isomorphism.  If $\sigma(H')\in K$,
  choose coordinates $z'_0,\dots,z'_4$ on $H'$ with $\sigma(H')=[1:0:\cdots:0]$ and
  $K=\{z'_4=0\}$; then $\rho$ is the substitution $z'_4=0$, its kernel
  $z'_4\,\calG_{\sigma,H'}$ has dimension four, and its image consists of the quadrics on $K$
  that are cones with vertex $\sigma(H')$.
\item If $\sigma(H)\notin K$, the projection from $\sigma(H)$ identifies $K$ with
  $\PP^3_{z_1,\dots,z_4}$ and $q|_K$ with $q$, which has rank four.  If $\sigma(H)\in K$,
  then $\lambda=\lambda(z_1,\dots,z_4)$ defines a plane $P_\lambda\subset\PP^3$, the
  subvariety $K$ is the cone over $P_\lambda$ with vertex $\sigma(H)$, and $q|_K$ is the cone
  over $q|_{P_\lambda}$, which has rank three unless $P_\lambda$ is tangent to $\Sigma$, in
  which case it has rank two.  Hence $q|_K$ is reducible only if $H'\in\Pi_H$ and
  $P_\lambda$ belongs to the dual surface of $\Sigma$; since $H'\mapsto P_\lambda$ is a
  morphism from $\Pi_H\setminus\{H\}$ to the $\dualPP^3$ of planes with one-dimensional
  fibers (the pencils through $K$), these $H'$ form a family of dimension at most three, and
  for them $q|_K=\lambda_1\lambda_2$ with $\lambda_1,\lambda_2$ non-proportional linear forms
  on $K$.
\end{enumerate}

\smallskip
\noindent\emph{(a) $s=0$.}  The stratum has dimension at most $\dim\mathbb Y_2=14$, and
$14+8=22$.

\smallskip
\noindent\emph{(b) $s=1$.}  Then $q|_K$ is reducible, so by (C) the hyperplane $H'$ varies
in a family of dimension at most three and $\sigma(H')\in K$ by (A).  Moreover $\rho(q')$ is
divisible by $\lambda_1$ or by $\lambda_2$, so $\rho(q')$ lies in one of two linear spaces of
dimension four; by (B), $q'$ lies in a linear space of dimension at most $4+4=8$, and $[q']$
varies in a family of dimension at most seven.  Thus $3+7+9=19$.

\smallskip
\noindent\emph{(c) $s=2$ and $q'|_K\ne0$.}  Then $\rho(q')\in\CC\,q|_K$.  If
$\sigma(H')\notin K$, then $\rho$ is injective by (B) and $[q']$ is determined by $H'$: the
stratum has dimension at most five.  If $\sigma(H')\in K$, then $H'\in\Pi_H$ varies in a
family of dimension four and $q'$ lies in a linear space of dimension at most $1+4$, so the
stratum has dimension at most $4+4=8$.  Thus $8+12=20$.

\smallskip
\noindent\emph{(d) $q'|_K=0$.}  Then $C'\supset K$, so the equation $\lambda'$ of $K$ in
$H'$ divides the cone form $q'$; hence $\lambda'$ is itself a cone form, i.e.\
$\sigma(H')\in K$, and $q'=\lambda' g$ with $[g]\in\PP(\calG_{\sigma,H'})\simeq\PP^3$.  The
stratum has dimension at most $4+3=7$, and $7+12=19$.

\smallskip
In every case the bound is smaller than $26=\dim L_C$.  Hence a general $F\in L_C$ contains
no involutive quadric cone other than $C$, and $\widetilde\rho_2$ is birational as in the
proof of Theorem~\ref{thm:higher-birationality-criterion}.
\end{proof}

\begin{theorem}[Birationality in \texorpdfstring{$\PP^5$}{P5}]
\label{thm:P5-birationality}
Let $n=3$, $m\ge2$ and $d\ge m$.  If $(m,d)\ne(2,2)$, then the involutive cone incidence
$\widetilde\rho_m:\PP(\widetilde{\mathcal F}^m_d)\to\mathbb B^m_d$ is birational.  For
$(m,d)=(2,2)$ the image is the full space of quadrics and the generic degree is six.
\end{theorem}

\begin{proof}
Corollary~\ref{cor:P5-birationality} covers $m=2,3$ with $d\ge4$, $m=4$ with $d\ge5$, and
$m\ge5$ with $d\ge m$.  Proposition~\ref{prop:boundary-birational} covers $(3,3)$ and $(4,4)$,
and Proposition~\ref{prop:P5-cubic-quadric} covers $(2,3)$.  The last assertion is
Proposition~\ref{prop:quadratic-general}.
\end{proof}

\section{Explicit formulas in \texorpdfstring{$\PP^3$}{P3}}
\label{sec:P3-formulas}

We now specialize to $n=2$, so the ambient space is $\PP^3$. In this case $\mathbb Y_m$
has dimension
$$
\dim\mathbb Y_m=m+3.
$$
The pullback by $s_\sigma$ corresponds, at the level of the Chern-class computation, to the
substitution
$$
\ell=h.
$$
Using \eqref{eq:segre-class} and the projective-bundle relations, one obtains the following
formulas.

\begin{theorem}
\label{thm:P3-formulas}
In $\PP^3$, except for the quadratic boundary case described in
Proposition~\ref{prop:quadratic-general}, the degrees of the varieties $\mathbb B^m_d$ of
surfaces containing an involutive cone of degree $m$ are as follows:
\begin{enumerate}[label=\textup{(\roman*)}, leftmargin=2em]
\item For $m=2$ and $d\ge3$,
$$
\begin{aligned}
\deg(\mathbb B^2_d)
={}&\frac{1}{288}d(d-1)(d^2-d+4)\\
&\times\left(3d^6-2d^5-9d^4+36d^3-72d^2+80d-48\right).
\end{aligned}
$$

\item For $m=3$ and $d\ge3$,
$$
\begin{aligned}
\deg(\mathbb B^3_d)
={}&\frac{1}{10368}(d-2)(d-1)
\bigl(81d^{10}-765d^9+3618d^8-9062d^7\bigr.\\
&\qquad\bigl.+5925d^6+44671d^5-194088d^4+426820d^3\\
&\qquad\bigl.-586512d^2+495360d-207360\bigr).
\end{aligned}
$$

\item For $m=4$ and $d\ge4$,
$$
\begin{aligned}
\deg(\mathbb B^4_d)
={}&\frac{1}{10368}(d-3)(d-2)(d^2-5d+12)\\
&\cdot\bigl(27d^{10}-441d^9+3366d^8-14926d^7+38239d^6\\
&\qquad -34201d^5-129764d^4+590888d^3\\
&\qquad -1162364d^2+1246584d-604800\bigr).
\end{aligned}
$$
\end{enumerate}
\end{theorem}

\begin{remark}
For $(m,d)=(2,2)$, the first formula gives the incidence degree $4$, not the degree of the
image. Proposition~\ref{prop:quadratic-general} shows that the image is the full space of
quadrics, hence has degree $1$. For all other cases covered above,
Theorem~\ref{thm:P3-birationality} identifies the incidence degree with the degree of the
image.
\end{remark}

\section{Explicit computations in \texorpdfstring{$\PP^5$}{P5}}
\label{sec:P5-computations}

In $\PP^5$ the cones are no longer pencils of lines.  We give the degrees in coefficient
form rather than as polynomials in $d$; even for $m=2$ the polynomial has degree $56$.

Let $n=3$.  On the involutive base, the pulled-back bundle $\calG$ has rank four and
$$
c(\calG)=\frac{1}{(1-h)(1+h)}=1+h^2+h^4
$$
modulo $h^6=0$.  Put
$$
r_m=\binom{m+3}{3}.
$$
For $m=2,3,4$, write
$$
c(\Sym^m\calG)=1+a_mh^2+b_mh^4.
$$
A computation with Chern roots gives
$$
(r_2,a_2,b_2)=(10,6,21),\qquad
(r_3,a_3,b_3)=(20,21,231),
$$
and
$$
(r_4,a_4,b_4)=(35,56,1596).
$$
Thus
$$
s(\Sym^m\calG)=1-a_mh^2+(a_m^2-b_m)h^4
$$
modulo $h^6$.

Set
$$
e_1=\binom{d+4}{5},\qquad
 e_m=\binom{d-m+5}{5},\qquad
 e_{m+1}=\binom{d-m+4}{5},
$$
and define
$$
S_{d,m}(h,x)=
\frac{1}{(1-h)^{e_1}}\cdot
\frac{(1-h-x)^{e_{m+1}}}{(1-x)^{e_m}}.
$$
For integers $p,q$, let
$$
C_{p,q}^{(m)}(d)=[h^px^q]S_{d,m}(h,x).
$$
Equivalently,
\begin{equation}
\label{eq:Cpq-finite-sum}
\begin{aligned}
C_{p,q}^{(m)}(d)
=\sum_{\alpha=0}^{p}\sum_{\beta=0}^{q}&(-1)^{\alpha+\beta}
\binom{e_1+p-\alpha-1}{p-\alpha}
\binom{e_{m+1}}{\alpha+\beta}\\
&\times\binom{\alpha+\beta}{\alpha}
\binom{e_m+q-\beta-1}{q-\beta}.
\end{aligned}
\end{equation}
with the convention that binomial coefficients outside their natural range are zero.

The incidence degree in $\PP^5$ is then
\begin{equation}
\label{eq:P5-degree-formula}
D_m(d)=
C_{5,r_m-1}^{(m)}(d)
-a_m C_{3,r_m+1}^{(m)}(d)
+(a_m^2-b_m)C_{1,r_m+3}^{(m)}(d).
\end{equation}
By Theorem~\ref{thm:P5-birationality} this is $\deg\mathbb B^m_d$ for $(m,d)\ne(2,2)$,
while $D_2(2)=6$ is the generic degree of Proposition~\ref{prop:quadratic-general}. The first
values are:
$$
\begin{array}{c|c|l}
 m & d & \deg(\mathbb B^m_d)\\
\hline
2 & 3 & 8\,885\,436\\
2 & 4 & 1\,311\,383\,133\,816\\
2 & 5 & 40\,592\,918\,532\,191\,508\\
3 & 3 & 20\,349\\
3 & 4 & 23\,210\,548\,536\\
3 & 5 & 119\,852\,371\,930\,648\,140\\
4 & 4 & 3\,819\,816\\
4 & 5 & 13\,399\,260\,046\,536\\
4 & 6 & 2\,703\,400\,709\,390\,290\,025\,480
\end{array}
$$

For $m=2$ the polynomial $D_2(d)$ factors as follows.

\begin{proposition}[Quadric cones in $\PP^5$]\label{prop:P5-quadric-expanded}
For $m=2$, put $\Delta=\binom{d+2}{4}=\rk\Sym^{d-2}\calQ$. The incidence degree $D_2(d)$ in
$\PP^5$ is
$$
D_2(d)=
\frac{1}{2^{36}\,3^{14}\,5^{6}\cdot 7\cdot 11\cdot 13}\;
\Delta\,(\Delta+5)(\Delta+6)(\Delta+7)(\Delta+8)\,P_{36}(d),
$$
where $P_{36}(d)$ is the primitive integer polynomial of degree $36$, with leading
coefficient $1\,049\,375$, supplied in \texttt{P36\_coefficients.txt} (Online Resource~1).
The polynomial $P_{36}$ is irreducible over $\mathbb Q$. For every $d\ge3$, $D_2(d)$ is the
degree of $\mathbb B^2_d$.
\end{proposition}

The factors $\Delta+5,\dots,\Delta+8$, and the factor
$d^2-5d+12=2\bigl(\binom{d-2}{2}+3\bigr)$ of $\deg\mathbb B^4_d$ in $\PP^3$, have a common
explanation.

\begin{remark}[Shifted-product factors]
\label{rem:shifted-factors}
Return to $\PP^{2n-1}$ and write $\Delta_h=e_1-e_{m+1}$, $\Delta_x=e_m-e_{m+1}=\rk\Sym^{d-m}\calQ$
and $E=e_{m+1}$, so that, by \eqref{eq:segre-class},
$$
\begin{aligned}
s(\widetilde{\mathcal F}^m_d)&=(1-h)^{-\Delta_h}(1-x)^{-\Delta_x}
\Bigl(1-\frac{hx}{(1-h)(1-x)}\Bigr)^{E}\\
&=\sum_{k\ge0}\binom{E}{k}(-1)^k\,h^kx^k\,(1-h)^{-\Delta_h-k}(1-x)^{-\Delta_x-k}.
\end{aligned}
$$
The coefficient of $h^px^q$ is therefore
$$
\sum_{k\ge0}\binom{E}{k}(-1)^k\binom{\Delta_h+p-1}{p-k}\binom{\Delta_x+q-1}{q-k},
\qquad
\binom{\Delta_x+q-1}{q-k}=\frac{1}{(q-k)!}\prod_{i=k}^{q-1}(\Delta_x+i).
$$
Only the monomials with $q\ge r_m-1$ and $p\le2n-1$ contribute to the integral over
$\mathbb Y_m$, and $k\le p$; hence every contributing term is divisible by
$\prod_{i=2n-1}^{r_m-2}(\Delta_x+i)$.  Moreover, as a polynomial in $d$, the integral vanishes
at $d=m-1,\dots,m-2n+2$, where $e_m=e_{m+1}=0$ and the Segre class does not involve $x$; these
are exactly the roots of $\Delta_x$.  Consequently
$$
\Delta_x\prod_{i=2n-1}^{r_m-2}(\Delta_x+i)\ \Big|\ \widetilde\delta_m(d)\deg(\mathbb B^m_d)
\qquad\text{in }\mathbb Q[d],
$$
and, since the incidence degree has degree at most $(2n-2)\dim\mathbb Y_m$ in $d$ (see
Section~\ref{sec:reproducible-computation}), the remaining factor has degree at most
$(2n-2)(4n-3)$, a bound independent of $m$.  In $\PP^5$ this accounts for the factors
$\Delta(\Delta+5)\cdots(\Delta+8)$ of $D_2$; for $m=3$ and $m=4$ the same computation, included
in Online Resource~1, shows that $D_m(d)$ is the product of
$\Delta_x\prod_{i=5}^{r_m-2}(\Delta_x+i)$ and a polynomial of degree $36$ which is again
irreducible over $\mathbb Q$.  In $\PP^3$ the argument explains the factor $\Delta_x+3$ of
$\deg\mathbb B^4_d$, but not the factor $d^2-d+4=2\bigl(\binom d2+2\bigr)$ of
$\deg\mathbb B^2_d$.
\end{remark}

\begin{remark}
Since the values in the table are nonzero, Corollary~\ref{cor:nonvanishing-finite} shows,
independently of Section~\ref{sec:higher-birationality}, that the incidence projections are
generically finite in these cases.
\end{remark}

\section{Computational verification}
\label{sec:reproducible-computation}

The formulas in Sections~\ref{sec:P3-formulas} and~\ref{sec:P5-computations} were checked with
exact arithmetic.  In $\PP^3$, after pulling back by the polarity section, one has
$$
c(\calG)=\frac{1}{(1-h)(1+h)}=\frac{1}{1-h^2},
\qquad c_1(\calG)=0,
\qquad c_2(\calG)=h^2.
$$
If $x=c_1(\calO_{\Sym^m\calG}(1))$, the calculation takes place in
$$
\mathbb Q[d,h,x]/(h^4,R_m),
$$
where
$$
R_m=c_{m+1}\bigl(\Sym^m\calG\otimes\calO_{\Sym^m\calG}(1)\bigr).
$$
The exact symbolic script expands every generalized binomial coefficient as a polynomial in
$d$, reduces the Segre class modulo $(h^4,R_m)$, and compares the resulting polynomial
with the closed formulas.  It returns
$$
\begin{array}{c|c|c}
 m & \deg_d & \text{symbolic difference}\\ \hline
2&10&0\\
3&12&0\\
4&14&0
\end{array}
$$
in $\mathbb Q[d]$.

The degree bound used in Section~\ref{sec:P5-computations} follows from the coefficient
formula. Set
$$
E=e_{m+1},\qquad \Delta_h=e_1-e_{m+1},\qquad \Delta_x=e_m-e_{m+1}=\binom{d-m+4}{4}.
$$
As polynomials in $d$, one has
$$
\deg_dE=5,\qquad \deg_d\Delta_h\le4,\qquad \deg_d\Delta_x\le4.
$$
Moreover,
$$
S_{d,m}(h,x)=
(1-h)^{-\Delta_h}(1-x)^{-\Delta_x}
\left(\frac{1-h-x}{(1-h)(1-x)}\right)^E.
$$
The last quotient satisfies
$$
\frac{1-h-x}{(1-h)(1-x)}
=1-\frac{hx}{(1-h)(1-x)},
$$
so its nonconstant part has total $(h,x)$-degree at least two. Consequently, its component
of total degree $t$ is a polynomial in $E$ of degree at most $\lfloor t/2\rfloor$. The
coefficient of $h^i$ in $(1-h)^{-\Delta_h}$ is
$\binom{\Delta_h+i-1}{i}$, and similarly for the $x$-factor; hence a component of total
degree $s$ in their product has degree in $d$ at most $4s$. Since the three coefficients
occurring in \eqref{eq:P5-degree-formula} have $p+q=r_m+4=\dim\mathbb Y_m$, a
contribution whose last factor has total degree $t$ has degree in $d$ at most
$$
4(r_m+4-t)+5\left\lfloor\frac{t}{2}\right\rfloor\le4(r_m+4).
$$
Thus $\deg_dD_m(d)\le4\dim\mathbb Y_m$, which is $56$, $96$ and $156$ for $m=2,3,4$; the
same argument in $\PP^{2n-1}$ gives the bound $(2n-2)\dim\mathbb Y_m$ used in
Remark~\ref{rem:shifted-factors}. For $m=2$ the factorized expression of
Proposition~\ref{prop:P5-quadric-expanded} also has degree at most
$$
4+4\cdot4+36=56.
$$
The verifier included in Online Resource~1 evaluates both the coefficient formula and the
factorized expression at the $57$ integers
$$
d=2,3,\ldots,58.
$$
All values agree exactly; the degree bound therefore proves the polynomial identity. The
coefficient file \texttt{P36\_coefficients.txt}, included in Online Resource~1, records the
$37$ integer coefficients of $P_{36}$; the verifier also checks that they are coprime and
that $P_{36}$ is irreducible over $\mathbb Q$, using the exact factorization routines of
SymPy. A second script recovers $D_m(d)$ for $m=2,3,4$ as a polynomial by exact
interpolation from $4\dim\mathbb Y_m+1$ integer values, divides it by
$\Delta_x\prod_{i=5}^{r_m-2}(\Delta_x+i)$, and factors the quotient, confirming the
statements of Remark~\ref{rem:shifted-factors}.

Online Resource~1 contains:
\begin{itemize}[leftmargin=2em]
\item \path{verify_involutive_cones_p3_exact.py}, which proves the three $\PP^3$ identities
  symbolically;
\item \path{compute_involutive_cones_p5.py}, which evaluates the compact $\PP^5$ formula
  with exact integer arithmetic;
\item \path{verify_p36_factor.py} and \path{P36_coefficients.txt}, which certify the
  factorization of Proposition~\ref{prop:P5-quadric-expanded};
\item \path{verify_shifted_factors.py}, which certifies the factorizations described in
  Remark~\ref{rem:shifted-factors} for $m=2,3,4$;
\item companion SageMath scripts and recorded Python outputs.
\end{itemize}

\section{Concluding remarks}

The polarity condition gives a particularly rigid compactification: the free point in the
ordinary cone incidence is replaced by the canonical section $H\mapsto\sigma(H)$.  In the
boundary case $d=m$, the condition that the polar hyperplane section have multiplicity $m$
at $\sigma(H)$ is equivalent, away from the locus of hypersurfaces containing $H$, to the
cone condition.  For $d>m$, global divisibility by a cone equation is strictly stronger than
the corresponding jet condition.  This suggests looking at involutive contact loci, which we
do not do here.

We do not know the exact birational range in $\PP^{2n-1}$ for $n\ge4$, nor what the
irreducible polynomials left over in Remark~\ref{rem:shifted-factors} mean geometrically.

\backmatter

\bmhead{Supplementary information}
\textbf{Online Resource 1.} Exact-arithmetic verification package for the formulas in
$\PP^3$ and $\PP^5$, containing Python and SageMath scripts, recorded Python outputs, and
the coefficient table for $P_{36}(d)$.

\bmhead{Acknowledgements}
The author thanks Professor Marcelo Pedro for his encouragement and collegial support. He
also thanks Professor Andr\'e Meireles for his guidance during the doctoral research and
Professor Israel Vainsencher for his decisive role in the formation of a generation of
algebraic geometers in Brazil.

\section*{Statements and Declarations}

\subsection*{Funding}
This research received no specific grant from any funding agency in the public, commercial,
or not-for-profit sectors.

\subsection*{Competing interests}
The author has no relevant financial or non-financial interests to disclose.

\subsection*{Author contributions}
The sole author conceived the study, developed the mathematics and software, performed the
computations and verification, and wrote and revised the manuscript.

\subsection*{Data availability}
No external datasets were used in this study. The exact-arithmetic scripts, coefficient data,
and recorded execution outputs supporting the computations are provided in Online
Resource~1.

\subsection*{Code availability}
The exact-arithmetic Python and SageMath source files and the execution outputs used to
reproduce the computational results are included in Online Resource~1.

\subsection*{Use of artificial intelligence}
The codimension-two converse, the incidence construction, and several low-dimensional
computations were developed in the author's 2014 doctoral thesis. During the preparation of
this article, ChatGPT (OpenAI) assisted with manuscript organization, language revision, and
the development and checking of symbolic-computation code. Claude (Anthropic) was used to
assist in updating and adapting the computational algorithms for implementation in Python and
SageMath, to audit the manuscript and cross-check the computations by independent
implementations, and to suggest the refinements that led to
Propositions~\ref{prop:boundary-birational} and~\ref{prop:P5-cubic-quadric} and to
Remark~\ref{rem:shifted-factors}, whose proofs the author verified. The author independently
reviewed and verified every mathematical statement, proof, reference, algorithm, computation,
and generated output, edited the final text, and takes full responsibility for the content of
the article.

\end{document}